\documentclass[12pt]{elsarticle}
\usepackage{graphicx}
\usepackage{amssymb}
\usepackage{amsthm}
\usepackage{epstopdf}
\usepackage{color}
\usepackage{mathtools} 
\usepackage{mathrsfs} 
\usepackage{xspace}
\usepackage{color,enumitem,graphicx}
\usepackage[colorlinks=true,urlcolor=blue,
citecolor=red,linkcolor=blue,linktocpage,pdfpagelabels,
bookmarksnumbered,bookmarksopen]{hyperref}
\usepackage[english]{babel}

\usepackage[top=1in, bottom=1in, left=1in, right=1in]{geometry}

\newfont{\script}{eusm10 scaled\magstep1}

\providecommand{\Real}{\mathop{\rm Re}\nolimits}%
\providecommand{\Imag}{\mathop{\rm Im}\nolimits}%

\usepackage{xcolor}
\usepackage{soul}
\definecolor{jddcol}{rgb}{0,0,0.8}
\definecolor{afcol}{rgb}{1,0,0}

\newtheorem{theorem}{Theorem}[section]

\newtheorem{lemma}[theorem]{Lemma}

\theoremstyle{definition}

\theoremstyle{remark}

\begin{document}

\begin{frontmatter}

\title{Fractional differential relations for the Lerch zeta function}

\date{}


\author[1]{Arran Fernandez\corref{cor1}}
\ead{arran.fernandez@emu.edu.tr}
\author[2]{Jean-Daniel Djida}
\ead{jeandaniel.djida@aims-cameroon.org}
\cortext[cor1]{Corresponding author}

\address[1]{\small Department of Mathematics, Eastern Mediterranean University, Famagusta, Northern Cyprus, via Mersin-10, Turkey}
\address[2]{{\small African Institute for Mathematical Sciences (AIMS), P.O. Box 608, Limbe Crystal Gardens, South West Region, Cameroon}}

\begin{abstract}
Starting from a recent result expressing the Lerch zeta function as a fractional derivative, we consider further fractional derivatives of the Lerch zeta function with respect to different variables. We establish a partial differential equation, involving an infinite series of fractional derivatives, which is satisfied by the Lerch zeta function.
\end{abstract}


\end{frontmatter}

\section{Introduction and Preliminaries}

The Riemann zeta function is one of the most important functions in the field of analytic number theory. It is the subject of the famous Riemann Hypothesis, as well as the Lindel\"of Hypothesis and several other unsolved mathematical problems, and its properties contain the key to describing the distribution of the prime numbers, which in turn is important in cryptography and computer security. For detailed discussion of this function and its significance, we refer to textbooks such as \cite{Titchmarsh,Edwards,Ivic}.

The Riemann zeta function is usually denoted by $\zeta(s)$ where $s=\sigma+it$ is a complex variable. In the right half plane $\sigma>1$, the function is defined by
\[
\zeta(s)=\sum_{m=0}^{\infty}m^{-s}\quad\text{ for }\Real(s)>1.
\]
Its definition for the remainder of the complex plane is given by analytic continuation in the variable $s$. There is a functional equation which relates $\zeta(s)$ with $\zeta(1-s)$ in an elementary way, and the most interesting and unknown range of values is given by $0<\sigma<1$, the so-called ``critical strip''.

There are many generalisations of the Riemann zeta function, such as the Dirichlet L-functions in number theory, and the Hurwitz and Lerch zeta functions which are again analytic functions of complex variables. We are going to focus on the latter functions, whose definitions we borrow from \cite{Lagarias2012}. The Hurwitz zeta function is defined by
\begin{equation}\label{Hurwitz:defn}
\zeta(x,s)=\sum_{m=0}^{\infty}(m+x)^{-s}\quad\text{ for }\Real(s)>1,~\Real(x)>0,
\end{equation}
and by analytic continuation for all $s\in\mathbb{C}$, while the Lerch zeta function is defined by
\begin{equation}\label{Lerch:defn}
L(t,x,s)=\sum_{m=0}^{\infty}(m+x)^{-s}e^{2\pi itm}\quad\text{ for }~\Real(s)>1,~\Real(x)>0,~\Imag(t)\geq 0,
\end{equation}
and by analytic continuation for $(t,x,s)$ in larger domains, extending to a universal cover of the manifold $\mathbb{C} \backslash\mathbb{Z}\times\mathbb{C}\backslash\mathbb{Z}_0^-\times\mathbb{C}$ (see \cite{Lagarias2012} for details). As one could have expected, the Riemann, Hurwitz, and Lerch zeta functions are related by the following identities: 
\[
\zeta(s)=\zeta(1,s);\quad\zeta(x,s)= L(0,x,s).
\]

Not only the Riemann zeta function, but also its extensions such as the Hurwitz and Lerch zeta functions, have been the subject of many analytic studies. In particular, a recent paper \cite{Arran-Lerch2018} of the first author established a new formula for the Lerch zeta function (hence also the Hurwitz and Riemann zeta functions as special cases) in terms of fractional calculus. This formula is reproduced in Theorem \ref{frac-Lerch} below, but first we shall introduce some main concepts and important properties of fractional calculus.\medskip

Fractional derivatives and integrals are the operators which emerge from the notion of the $n$th derivative or $n$th repeated integral when $n$ is no longer an integer. There are several different proposals for how to define these operators, which can be classified into broad classes of ``fractional operators'' \cite{baleanu-fernandez}, including integral operators with various kernel functions \cite{fernandez-ozarslan-baleanu} and differentiation operators with respect to monotonic functions \cite[\S18]{samko-kilbas-marichev}, among others.

We shall focus here on the definition of fractional calculus which is generally seen as the most fundamental and important, namely the Riemann--Liouville definition. Then, fractional integrals are defined by
\[
\prescript{}{a}I^{\nu}_tf(t)=\frac{1}{\Gamma(\nu)}\int_a^t(t-\tau)^{\nu-1}f(\tau)\,\mathrm{d}t\quad\text{ for }~\Real(\nu)>0,
\]
where $a$ is a constant of integration. (In this paper we shall usually take $a=-\infty$ for a semi-infinite interval of integration.) Fractional derivatives are defined by
\[
\prescript{}{a}D^{\nu}_tf(t)=\frac{\mathrm{d}^n}{\mathrm{d}t^n}\prescript{}{a}I^{n-\nu}_tf(t)\quad\text{ for }~\Real(\nu)\geq0,\quad n-1<\Real(\nu)\leq n,
\]
and we shall use the convention that $\prescript{}{a}D^{-\nu}_tf(t)=\prescript{}{a}I^{\nu}_tf(t)$ so that the expression $\prescript{}{a}D^{\nu}_tf(t)$ is defined for all $\nu$ in the whole complex plane. We sometimes call this expression a fractional differintegral, this term covering both derivatives and integrals.\medskip

It is also possible to define the right-handed version of Riemann--Liouville fractional calculus, which can be viewed as fractional powers of the operators $\int_t^b$ and $-\frac{\mathrm{d}}{\mathrm{d}t}$ instead of the operators $\int_a^t$ and $\frac{\mathrm{d}}{\mathrm{d}t}$. In this case, the formulae are
\[
\prescript{}{t}I^{\nu}_bf(t)=\frac{1}{\Gamma(\nu)}\int_t^b(\tau-t)^{\nu-1}f(\tau)\,\mathrm{d}t\quad\text{ for }~\Real(\nu)>0,
\]
\[
\prescript{}{t}D^{\nu}_bf(t)=(-1)^n\frac{\mathrm{d}^n}{\mathrm{d}t^n}\prescript{}{t}I^{n-\nu}_bf(t)\quad\text{ for }~\Real(\nu)\geq0,\quad n-1<\Real(\nu)\leq n,
\]
where $b$ is a constant of integration. In the special case that $b=\infty$, for a semi-infinite interval of integration, this type of fractional calculus is called Weyl fractional calculus \cite{miller-ross}. There is a simple relationship between $\prescript{}{-\infty}D_t^{\nu}$ and $\prescript{}{t}D_{\infty}^{\nu}$ which is given by changing the sign of the variable $t$ \cite{miller-ross}. \medskip

The operators of Riemann--Liouville and Weyl fractional calculus have many interesting properties, and their theory is covered in depth in several textbooks such as \cite{samko-kilbas-marichev, miller-ross,oldham-spanier}. In particular, we would like to mention the generalised Leibniz rule for fractional derivatives and integrals, as follows.

\begin{theorem}[\cite{Watanabe,Osler}]
\label{Thm:Leibniz}
If $f(z)$ and $g(z)$ are two analytic functions on $\mathbb{C}$, then the fractional differintegral of their product is given by
\begin{equation}
\label{Leibniz:eqn}
\prescript{}{a}D^{\nu}_z\big(f(z)g(z)\big)=\sum_{n=-\infty}^{\infty}\binom{\nu}{\mu+n}\prescript{}{a}D^{\nu-\mu-n}_zf(z)\prescript{}{a}D^{\mu+n}_zg(z),
\end{equation}
for any $\mu,\nu\in\mathbb{C}$ with $\nu\not\in\mathbb{Z}^-$.
\end{theorem}

As a special case of this Theorem, putting $\mu=0$, we get the most natural generalisation of the original Leibniz rule for repeated derivatives:
\begin{equation}
\label{Leibniz:eqn:special}
\prescript{}{a}D^{\nu}_z\big(f(z)g(z)\big)=\sum_{n=0}^{\infty}\binom{\nu}{n}\prescript{}{a}D^{\nu-n}_zf(z)\prescript{}{a}D^{n}_zg(z).
\end{equation}
The symmetry between $f$ and $g$ in this expression \eqref{Leibniz:eqn:special}, however, is not clear, while it is clear in the more general expression \eqref{Leibniz:eqn}. \medskip

Additionally, the same results \eqref{Leibniz:eqn} and \eqref{Leibniz:eqn:special} are also valid for the right-handed differintegrals $\prescript{}{t}D_b^{\nu}$, and in particular for the Weyl fractional calculus given by $b=\infty$. \medskip

In a recent paper of the first author \cite{Arran-Lerch2018}, a new connection was established between zeta functions and fractional calculus. The first such connection was discovered by J. B. Keiper in 1975 \cite{Keiper}, and further connections have been studied by others, but the paper \cite{Arran-Lerch2018} was the first to discover an expression for the Riemann zeta function $\zeta(s)$ as a fractional differintegral which is valid even for $s$ in the critical strip $0<\Real(s)<1$. Here, we use these results as a starting point and continue with further investigation of the Lerch zeta function from the viewpoint of fractional calculus. \medskip

In Section \ref{sec:correc}, we revisit the results of \cite{Arran-Lerch2018}, provide a correction to the domain of validity, and prove a related result involving Weyl fractional calculus. In Section \ref{sec:tder} and Section \ref{sec:xder}, we consider partial fractional differintegration with respect to $t$ and $x$ respectively of the formula \eqref{result:eqn}. In Section \ref{sec:pde}, we combine the results of the previous two sections to obtain an infinite-order partial fractional differential equation which is satisfied by the Lerch zeta function.

\section{Revisiting the Results of \cite{Arran-Lerch2018}} \label{sec:correc}

It is necessary to correct an omission in the paper \cite{Arran-Lerch2018} in the assumptions on certain variables. The correct statement of \cite[Lemma 1.2]{Arran-Lerch2018} is as follows.

\begin{lemma}
\label{RL:exp}
For $\alpha,k\in\mathbb{C}$ with $\Real(k)\geq0$,
\begin{equation}
\label{RL:exp:eqn}
\prescript{}{-\infty}D_{t}^{\alpha}\left(e^{kt}\right)=k^{\alpha}e^{kt},
\end{equation}
where complex power functions are defined by the principal branch.
\end{lemma}

\begin{proof}
We first assume $-1<\Real(\alpha)<0$, an assumption we will remove later by analytic continuation. Then, substituting $v=kt-k\tau$, we have
\begin{align*}
\prescript{}{-\infty}D_{t}^{\alpha}\left(e^{kt}\right)&=\tfrac{1}{\Gamma(-\alpha)}\int_{-\infty}^t(t-\tau)^{-\alpha-1}e^{k\tau}\,\mathrm{d}\tau \\
&=\tfrac{1}{\Gamma(-\alpha)}\int_{k\infty}^0\left(\tfrac{v}{k}\right)^{-\alpha-1}e^{kt-v}\big(\tfrac{1}{-k}\big)\,\mathrm{d}v \\
&=\tfrac{k^{\alpha}e^{kt}}{\Gamma(-\alpha)}\int^{k\infty}_0v^{-\alpha-1}e^{-v}\,\mathrm{d}v.
\end{align*}
This integral can be deformed to the usual gamma-function integral from $0$ to $\infty$, but only under the assumption that $\Real(k)>0$, or, by Jordan's lemma, $\Real(k)=0$.
\end{proof}

Due to the extra assumption required in the above Lemma, the main result of \cite{Arran-Lerch2018} also requires an extra assumption on the variable $x$. We state the correct form as follows, as it is our main starting point for this paper.


\begin{theorem}[\cite{Arran-Lerch2018}]
\label{frac-Lerch}
The Lerch zeta function can be written as
\begin{equation}\label{result:eqn}
L(t,x,s)=(2\pi)^s\exp\left[i\pi(\tfrac{s}{2}-2tx)\right]\prescript{}{-\infty}D_{t}^{-s}\left(\frac{e^{2\pi itx}}{1-e^{2\pi it}}\right),
\end{equation}
for any complex numbers $s,x,t$ satisfying $\Imag(t)>0$ and $\Imag(x)\leq0$, $x\not\in(-\infty,0]$.
\end{theorem}

In the original paper \cite{Arran-Lerch2018}, this result was stated without the $\Imag(x)\leq0$ assumption. In the case $\Imag(x)>0$, however, we do have an analogous result, which involves right-handed differintegrals or Weyl differintegrals. This is derived as follows, starting with a lemma analogous to Lemma \ref{RL:exp} above.

\begin{lemma}
\label{Weyl:exp}
For $\alpha,k\in\mathbb{C}$ with $\Real(k)\leq0$,
\begin{equation}
\label{Weyl:exp:eqn}
\prescript{}{t}D_{\infty}^{\alpha}\left(e^{kt}\right)=(-k)^{\alpha}e^{kt},
\end{equation}
where complex power functions are defined by the principal branch.
\end{lemma}

\begin{proof}
We first assume $-1<\Real(\alpha)<0$, an assumption we will remove later by analytic continuation. Then, substituting $v=kt-k\tau$ as before, we have
\begin{align*}
\prescript{}{t}D_{\infty}^{\alpha}\left(e^{kt}\right)&=\tfrac{1}{\Gamma(-\alpha)}\int_t^{\infty}(\tau-t)^{-\alpha-1}e^{k\tau}\,\mathrm{d}\tau \\
&=\tfrac{1}{\Gamma(-\alpha)}\int_0^{-k\infty}\left(\tfrac{v}{-k}\right)^{-\alpha-1}e^{kt-v}\big(\tfrac{1}{-k}\big)\,\mathrm{d}v \\
&=\tfrac{(-k)^{\alpha}e^{kt}}{\Gamma(-\alpha)}\int_0^{-k\infty}v^{-\alpha-1}e^{-v}\,\mathrm{d}v.
\end{align*}
This integral can be deformed to the usual gamma-function integral from $0$ to $\infty$, but only under the assumption that $\Real(k)<0$, or, by Jordan's lemma, $\Real(k)=0$.
\end{proof}


\begin{theorem}
\label{Weyl:Lerch}
The Lerch zeta function can be written as
\begin{equation}\label{Weyl:result:eqn}
L(t,x,s)=(2\pi)^s\exp\left[-i\pi(\tfrac{s}{2}+2tx)\right]\prescript{}{t}D_{\infty}^{-s}\left(\frac{e^{2\pi itx}}{1-e^{2\pi it}}\right),
\end{equation}
for any complex numbers $s,x,t$ satisfying $\Imag(t)>0$ and $\Imag(x)\geq0$, $x\not\in(-\infty,0]$.
\end{theorem}

\begin{proof}
The method of proof is exactly the same as in \cite{Arran-Lerch2018}. We start from the series definition \eqref{Lerch:defn} of the Lerch zeta function and follow the argument of \cite[Theorem 2.1]{Arran-Lerch2018}:
\begin{align*}
L(t,x,s)&=\sum_{m=0}^{\infty}(m+x)^{-s}e^{2\pi itm}=(-2\pi i)^s\sum_{m=0}^{\infty}(-2\pi im-2\pi ix)^{-s}e^{2\pi itm} \\
&=(-2\pi i)^se^{-2\pi itx}\sum_{m=0}^{\infty}(-2\pi i(m+x))^{-s}e^{2\pi it(m+x)} \\
&=(-2\pi i)^se^{-2\pi itx}\sum_{m=0}^{\infty}\prescript{}{t}D^{-s}_{\infty}\left(e^{2\pi it(m+x)}\right) \\
&=(-2\pi i)^se^{-2\pi itx}\prescript{}{t}D^{-s}_{\infty}\left(\frac{e^{2\pi it(m+x)}}{1-e^{2\pi it}}\right),
\end{align*}
provided that $\Imag(x)\geq0$ for the Weyl differintegration of the exponential function. The result is now proved under the assumptions $\Real(s)>1$, $\Real(x)>0$, $\Imag(x)\geq0$, and $\Imag(t)\geq0$. The remaining argument, using analytic continuation to remove these assumptions, is exactly as in \cite{Arran-Lerch2018}.
\end{proof}

The original work of \cite{Arran-Lerch2018} demonstrated that the formula \eqref{result:eqn} is consistent with certain other properties of the Lerch zeta function. Now, we shall take this work to the next stage by using the result of Theorem \ref{frac-Lerch} to prove further results on fractional derivatives of the Lerch zeta function, which can be written in the form of fractional partial differential equations satisfied by this function.

\section{Partial Derivatives with respect to $\boldsymbol{t}$} \label{sec:tder}

Before stating the main result of this section, we shall motivate it by deriving it in an unrigorous way using the series \eqref{Lerch:defn} for the Lerch zeta function. Our aim here is to find an expression for the fractional derivative with respect to $t$ of the Lerch zeta function $L(t,x,s)$. Starting from the series \eqref{Lerch:defn}, we can think in the following way:
\begin{align*}
\prescript{}{-\infty}D_{t}^{\alpha} L(t,x,s)&=\prescript{}{-\infty}D_{t}^{\alpha}\left(\sum_{m=0}^{\infty}(m+x)^{-s}e^{2\pi itm}\right) \\
&=\sum_{m=0}^{\infty}(m+x)^{-s}\prescript{}{-\infty}D_{t}^{\alpha}\left(e^{2\pi itm}\right) \\
&=\sum_{m=0}^{\infty}(m+x)^{-s}(2\pi im)^{\alpha}e^{2\pi itm} \\
&=(2\pi i)^{\alpha}\sum_{m=0}^{\infty}(m+x)^{-s+\alpha}\left(\frac{m}{m+x}\right)^{\alpha}e^{2\pi itm}.
\end{align*}
Using the binomial theorem, we have
\begin{align*}
\left(\frac{m}{m+x}\right)^{\alpha}=\left(1-\frac{x}{m+x}\right)^{\alpha}=\sum_{n=0}^{\infty}\binom{\alpha}{n}(-x)^n(m+x)^{-n},
\end{align*}
and therefore
\begin{align*}
\prescript{}{-\infty}D_{t}^{\alpha} L(t,x,s)&=(2\pi i)^{\alpha}\sum_{m=0}^{\infty}\sum_{n=0}^{\infty}\binom{\alpha}{n}(m+x)^{-s+\alpha-n}(-x)^ne^{2\pi itm} \\
&=(2\pi i)^{\alpha}\sum_{n=0}^{\infty}\binom{\alpha}{n}(-x)^n\sum_{m=0}^{\infty}(m+x)^{-s+\alpha-n}e^{2\pi itm} \\
&=(2\pi i)^{\alpha}\sum_{n=0}^{\infty}\binom{\alpha}{n}(-x)^nL(t,x,s-\alpha+n).
\end{align*}
The above derivation is neat but not rigorous. We have interchanged the order of fractional differintegration and summation; we have interchanged the two infinite series over $m$ and $n$; and we have not considered the restrictions on parameters necessary for all these series to converge. The final result, however, is valid, and more easily proved by using Theorem \ref{frac-Lerch}, as follows.

\begin{theorem}\label{theo:t}
For any $\alpha\in\mathbb{C}$, the Lerch zeta function satisfies the identity
\[
\prescript{}{-\infty}D_{t}^{\alpha} L(t,x,s)=(2\pi i)^{\alpha}\sum_{n=0}^{\infty}\binom{\alpha}{n}(-x)^nL(t,x,s-\alpha+n),
\]
for $\Imag(t)<0$, $\Imag(x)\leq0$, $x\not\in(-\infty,0]$, and $\Real(s)>0$, and also for any $t,x,s$ such that both sides converge.
\end{theorem}

\begin{proof}
We start from the expression \eqref{result:eqn} for the Lerch zeta function, and use the fractional Leibniz rule \eqref{Leibniz:eqn} for fractional differintegrals of products:
\[
\begin{aligned}
\prescript{}{-\infty}D_{t}^{\alpha} L(t,x,s)&= (2\pi)^s \prescript{}{-\infty}D_{t}^{\alpha} \left[e^{i\pi(\tfrac{s}{2}-2tx)}\prescript{}{-\infty}D_{t}^{-s}\left(\frac{e^{2\pi itx}}{1-e^{2\pi it}}\right)\right]\\
&=(2\pi)^{s}e^{i\pi s/2} \sum_{n=0}^{\infty} \binom{\alpha}{n}\frac{\mathrm{d}^n}{\mathrm{d}t^{n}} \left( e^{-2\pi itx}\right)\prescript{}{-\infty}D_{t}^{\alpha-n} \left[ \prescript{}{-\infty}D_{t}^{-s}\left(\frac{e^{2\pi itx}}{1-e^{2\pi it}}\right)\right]\\
&=(2\pi i)^{s} \sum_{n=0}^{\infty} \binom{\alpha}{n}\left(-2\pi ix \right)^{n} e^{-2\pi itx}\prescript{}{-\infty}D_{t}^{-s+\alpha-n}\left(\frac{e^{2\pi itx}}{1-e^{2\pi it}}\right)\\
&=\sum_{n=0}^{\infty} \binom{\alpha}{n}(2\pi i)^{s-\alpha+n} (2\pi i)^{\alpha-n} \left(-2\pi ix \right)^{n} e^{-2\pi itx}\prescript{}{-\infty}D_{t}^{-(s-\alpha+n)}\left(\frac{e^{2\pi itx}}{1-e^{2\pi it}}\right)\\
&=\sum_{n=0}^{\infty} \binom{\alpha}{n}(2\pi i)^{\alpha}(-x)^{n}\left[(2\pi)^{s-\alpha+n} e^{i\pi(\frac{s-\alpha+n}{2}-2tx)}\prescript{}{-\infty}D_{t}^{-(s-\alpha+n)}\left(\frac{e^{2\pi itx}}{1-e^{2\pi it}}\right)\right] \\
&=(2\pi i)^{\alpha}\sum_{n=0}^{\infty} \binom{\alpha}{n}(-x)^{n}L(t,x,s-\alpha+n).
\end{aligned}
\]
Here we have assumed that $\Real(s)>0$, so that the composition of fractional differintegrals works smoothly without extra initial value terms, as well as the restrictions on $t$ and $x$ from Theorem \ref{frac-Lerch}. But these conditions can be removed by analytic continuation, provided that both sides of the identity are well-defined and analytic.
\end{proof}

\begin{theorem}\label{theo:t2}
For any $\alpha\in\mathbb{C}$, the Lerch zeta function satisfies the identity
\[
\prescript{}{t}D_{\infty}^{\alpha} L(t,x,s)=(-2\pi i)^{\alpha}\sum_{n=0}^{\infty}\binom{\alpha}{n}(-x)^nL(t,x,s-\alpha+n),
\]
for $\Imag(t)<0$, $\Imag(x)\geq0$, $x\not\in(-\infty,0]$, and $\Real(s)>0$, and also for any $t,x,s$ such that both sides converge.
\end{theorem}

\begin{proof}
The argument goes in exactly the same way as for Theorem \ref{theo:t}, with all powers of $2\pi i$ replaced by the corresponding powers of $-2\pi i$ and vice versa.
\end{proof}

\section{Partial Derivatives with respect to $\boldsymbol{x}$} \label{sec:xder}

\begin{lemma}
For any $k\in\mathbb{N}$ and $t,x,s$ in the domain of analyticity of the Lerch zeta function
\[
\frac{\partial^k}{\partial x^k} L(t,x,s)=\frac{\Gamma(1-s)}{\Gamma(1-s-k)}L(t,x,s+k).
\]
\end{lemma}

\begin{proof}
First of all, we note that applying the fractional Leibniz rule \eqref{Leibniz:eqn} with $\mu=0$ and $g(z)=1$ gives the following neat identity:
\begin{equation}
\label{Leibniz:specialcase}
\prescript{}{a}D^{\nu}_z\big(zf(z)\big)=\sum_{n=0}^{\infty}\binom{\nu}{n}\prescript{}{a}D^{\nu-n}_zf(z)\prescript{}{a}D^{n}_z(z)=z\big[\prescript{}{a}D^{\nu}_zf(z)\big]+\nu\big[\prescript{}{a}D^{\nu-1}_zf(z)\big].
\end{equation}

For the proof, we start with $k=1$ and use the expression \eqref{result:eqn} for the Lerch zeta function:
\begin{align*}
\frac{\partial}{\partial x} L(t,x,s)&=(2\pi i)^s\frac{\partial}{\partial x}\left[e^{-2i\pi tx}\prescript{}{-\infty}D_{t}^{-s}\left(\frac{e^{2\pi itx}}{1-e^{2\pi it}}\right)\right] \\
&=(2\pi i)^s\left[(-2i\pi t)e^{-2i\pi tx}\prescript{}{-\infty}D_{t}^{-s}\left(\frac{e^{2\pi itx}}{1-e^{2\pi it}}\right)+e^{-2i\pi tx}\prescript{}{-\infty}D_{t}^{-s}\left(\frac{2\pi ite^{2\pi itx}}{1-e^{2\pi it}}\right)\right] \\
&=(2\pi i)^se^{-2i\pi tx}\left[(-2i\pi t)\prescript{}{-\infty}D_{t}^{-s}\left(\frac{e^{2\pi itx}}{1-e^{2\pi it}}\right)+\right. \\
&\hspace{2.5cm}\left.2\pi i\left[(t)\prescript{}{-\infty}D_{t}^{-s}\left(\frac{e^{2\pi itx}}{1-e^{2\pi it}}\right)+(-s)\prescript{}{-\infty}D_{t}^{-s-1}\left(\frac{e^{2\pi itx}}{1-e^{2\pi it}}\right)\right]\right] \\
&=(-s)(2\pi i)^se^{-2i\pi tx}\prescript{}{-\infty}D_{t}^{-s-1}\left(\frac{e^{2\pi itx}}{1-e^{2\pi it}}\right) \\
&=-sL(t,x,s+1),
\end{align*}
where in the third line we used the identity \eqref{Leibniz:specialcase}. The result for general $k$ follows by finite descent:
\begin{align*}
\frac{\partial^k}{\partial x^k} L(t,x,s)&=\frac{\partial^{k-1}}{\partial x^{k-1}}\big[(-s) L(t,x,s+1)\big]=\frac{\partial^{k-2}}{\partial x^{k-2}}\big[(-s)(-s-1) L(t,x,s+2)\big] \\
&=\dots=(-s)(-s-1)(-s-2)\dots(-s-k+1)L(t,x,s+k) \\
&=\frac{\Gamma(1-s)}{\Gamma(1-s-k)}L(t,x,s+k).
\end{align*}
The above argument relies on the result of Theorem \ref{frac-Lerch}, with the associated restrictions on $t,x,s$. However, these restrictions can be removed for the final result by analytic continuation.
\end{proof}

\begin{theorem}\label{theo:x}
For any $\alpha\in\mathbb{C}$, the Lerch zeta function satisfies the identity
\[
\prescript{}{x}D_{\infty}^{\alpha}~ L(t,x,s)=\frac{\Gamma(1-s)}{\Gamma(1-s-\alpha)}e^{-i\pi\alpha} L(t,x,s+\alpha),
\]
for any $t,x,s$ such that this fractional derivative exists.
\end{theorem}

\begin{proof}
We start from \eqref{result:eqn} as usual, and use the general fractional Leibniz rule \eqref{Leibniz:eqn} twice with respect to $x$ in the following manipulations:
\[
\begin{aligned}
\prescript{}{x}D_{\infty}^{\alpha}~ L(t,x,s) &= (2\pi i)^s \prescript{}{x}D_{\infty}^{\alpha} \left[e^{-2\pi itx}\prescript{}{-\infty}D_{t}^{-s}\left(\frac{e^{2\pi itx}}{1-e^{2\pi it}}\right)\right]\\
&=(2\pi i)^{s}\sum_{n=-\infty}^{\infty} \binom{\alpha}{\gamma+n}\prescript{}{x}D_{\infty}^{\alpha-\gamma-n}\left( e^{-2\pi itx}\right)\prescript{}{-\infty}D_{t}^{-s} \prescript{}{x}D_{\infty}^{\gamma+n}\left(\frac{e^{2\pi itx}}{1-e^{2\pi it}}\right)\\
&=(2\pi i)^{s}\sum_{n=-\infty}^{\infty} \binom{\alpha}{\gamma+n}(2\pi it)^{\alpha-\gamma-n}e^{-2\pi itx}
\prescript{}{-\infty}D_{t}^{-s} \left((-2\pi it)^{\gamma+n}\frac{e^{2\pi itx}}{1-e^{2\pi it}}\right)\\
&=(2\pi)^{s+\alpha}\sum_{n=-\infty}^{\infty} \binom{\alpha}{\gamma+n}i^{s+\alpha-2\gamma-2n} t^{\alpha-\gamma-n}e^{-2\pi itx}\\
&\hspace{2cm} \times \left[\sum_{k=-\infty}^{\infty} \binom{-s}{\beta + k}
\prescript{}{-\infty}D_{t}^{\beta + k} \left(t^{\gamma+n}\right) \prescript{}{-\infty}D_{t}^{-s-\beta-k}  \left(\frac{e^{2\pi itx}}{1-e^{2\pi it}}\right)\right],
\end{aligned}
\]
where $\beta$ and $\gamma$ are the arbitrary parameters from Theorem \ref{Thm:Leibniz} (denoted there as $\mu$). The free choice of these two parameters will enable us to prove the theorem.

Note that in the above manipulation we have assumed that the fractional differintegration of exponential functions works in the way of Lemma \ref{Weyl:exp}. This assumption will be justified later, after we have fixed the values of $\beta$ and $\gamma$.

Firstly, we choose $\beta = \gamma + n$:
\begin{multline*}
\prescript{}{x}D_{\infty}^{\alpha}~ L(t,x,s) =(2\pi)^{s+\alpha}\sum_{n=-\infty}^{\infty} \binom{\alpha}{\gamma+n}i^{s+\alpha-2\gamma-2n}e^{-2\pi itx}t^{\alpha-\gamma-n}e^{-2\pi itx}\\
\times \sum_{k=-\infty}^{0} \binom{-s}{\gamma + n + k}
\left[\frac{\Gamma(\gamma+n+1)}{\Gamma(1-k)}t^{-k}\right]  \prescript{}{-\infty}D_{t}^{-s-\gamma-n-k}\left(\frac{e^{2\pi itx}}{1-e^{2\pi it}}\right)
\end{multline*}
Swapping $k$ with $-k$, we get the following:
\begin{multline*}
\prescript{}{x}D_{\infty}^{\alpha}~ L(t,x,s) =(2\pi)^{s+\alpha}\sum_{n=-\infty}^{\infty} \sum_{k=0}^{\infty} 
i^{s+\alpha-2\gamma-2n}t^{\alpha-\gamma-n+k}e^{-2\pi itx} \\
\times \frac{\alpha!(-s)!}{(\alpha-\gamma-n)!(\gamma + n-k)!(-s-\gamma-n+k)!k!} \prescript{}{-\infty}D_{t}^{-s-\gamma-n+k}\left(\frac{e^{2\pi itx}}{1-e^{2\pi it}}\right).
\end{multline*}

Next, we choose $\gamma = \alpha$:
\begin{multline*}
\prescript{}{x}D_{\infty}^{\alpha}~ L(t,x,s) =(2\pi)^{s+\alpha}\sum_{n=-\infty}^{0} \sum_{k=0}^{\infty} i^{s-\alpha-2n}t^{-n+k}e^{-2\pi itx} \\
\times \frac{\alpha!(-s)!}{(-n)!(\alpha + n-k)!(-s-\alpha-n+k)!k!} \prescript{}{-\infty}D_{t}^{-s-\alpha-n+k}\left(\frac{e^{2\pi itx}}{1-e^{2\pi it}}\right).
\end{multline*}
Swapping $n$ with $-n$, we get the following:
\begin{multline*}
\prescript{}{x}D_{\infty}^{\alpha}~ L(t,x,s) =(2\pi)^{s+\alpha}i^{s-\alpha}\sum_{n=0}^{\infty} \sum_{k=0}^{\infty} (-1)^nt^{n+k}e^{-2\pi itx} \\
\times \frac{\alpha!(-s)!}{n!(\alpha-n-k)!(-s-\alpha+n+k)!k!} \prescript{}{-\infty}D_{t}^{-s-\alpha+n+k}\left(\frac{e^{2\pi itx}}{1-e^{2\pi it}}\right).
\end{multline*}
Setting $m=n+k$ to simplify:
\[
\begin{aligned}
\prescript{}{x}D_{\infty}^{\alpha}~ L(t,x,s) &=(2\pi i)^{s+\alpha}e^{-i\pi\alpha}\sum_{m=0}^{\infty} \left[\sum_{n=0}^{m}\frac{(-1)^{n}m!}{n!(m-n)!}\right]t^{m}e^{-2\pi itx}\\
&\times \frac{\alpha!(-s)!}{m!(\alpha-m)!(-s-\alpha + m)!} \prescript{}{-\infty}D_{t}^{-s-\alpha+m}\left(\frac{e^{2\pi itx}}{1-e^{2\pi it}}\right)
\end{aligned}
\]
But since
\[
\sum_{n=0}^{m}\frac{(-1)^{n}m!}{n!(m-n)!} = \sum_{n=0}^{m} \binom{m}{n}(-1)^{n}1^{m-n} = (1-1)^{m} = \delta_{0m},
\]
we have that
\[
\begin{aligned}
\prescript{}{x}D_{\infty}^{\alpha}~ L(t,x,s) &=(2\pi i)^{s+\alpha}e^{-i\pi\alpha}t^{0}e^{-2\pi itx}\frac{(-s)!}{(-s-\alpha)!} \prescript{}{-\infty}D_{t}^{-s-\alpha}\left(\frac{e^{2\pi itx}}{1-e^{2\pi it}}\right)\\
&=e^{-i\pi\alpha}\frac{(-s)!}{(-s-\alpha)!} L(t,x,s+\alpha).
\end{aligned}
\]
We now note that, at the beginning of the proof, the differintegral to order $\alpha-\gamma-n$ was actually a standard repeated derivative, since $\alpha=\gamma$ and $n\in\mathbb{Z}^-$ (before the swapping of $n$ with $-n$). Meanwhile, the differintegration to order $\gamma+n$ was valid by Lemma \ref{Weyl:exp}, because $\Real(2\pi it)<0$ by the assumptions of Theorem \ref{frac-Lerch}. So the result is valid.

Again, the conditions on $t,x,s$ which are required by Theorem \ref{frac-Lerch} can be removed by analytic continuation at the end of the proof.
\end{proof}

\section{A Fractional Infinite-Order Partial Differential Equation} \label{sec:pde}

Combining the results of Theorem \ref{theo:t} and Theorem \ref{theo:x}, we obtain the following result on the Lerch zeta function.

\begin{theorem}
The Lerch zeta function satisfies the following infinite-order partial fractional differential equation in $t$ and $x$:
\[
\prescript{}{-\infty}D_{t}^{\alpha} L(t,x,s)=(-2\pi i)^{\alpha}\sum_{n=0}^{\infty}x^n\frac{\alpha!(\alpha-s-n)!}{(\alpha-n)!n!(-s)!}\prescript{}{x}D_{\infty}^{n-\alpha}~ L(t,x,s),
\]
or, using Weyl differintegrals only,
\[
\prescript{}{t}D_{\infty}^{\alpha} L(t,x,s)=(2\pi)^{\alpha}(-i)^{3\alpha}\sum_{n=0}^{\infty}x^n\frac{\alpha!(\alpha-s-n)!}{(\alpha-n)!n!(-s)!}\prescript{}{x}D_{\infty}^{n-\alpha}~ L(t,x,s).
\]
\end{theorem}

\begin{proof}
By Theorem \ref{theo:x}, we have
\[
e^{i\pi(n-\alpha)}\frac{\Gamma(1-s+\alpha-n)}{\Gamma(1-s)}\prescript{}{x}D_{\infty}^{n-\alpha}~ L(t,x,s)= L(t,x,s+n-\alpha).
\]
Substituting this into the results of Theorem \ref{theo:t} and Theorem \ref{theo:t2} yields, respectively, the two stated identities.
\end{proof}

\section{Conclusions}

This paper has been a continuation of the work of \cite{Arran-Lerch2018}, whose main result we cited as our Theorem \ref{frac-Lerch} above. We have added an extra condition in this Theorem which was erroneously omitted in the corresponding result in \cite{Arran-Lerch2018}; however, this does not affect the validity of the expression. Starting from this Theorem and the closely related Theorem \ref{Weyl:Lerch} which uses Weyl fractional calculus, we have considered fractional partial derivatives of the Lerch zeta function with respect to both $t$ and $x$, and in each case we discovered some functional equation relationships. Combining the results of these partial differentiation studies with respect to $t$ and $x$, we produced a fractional partial differential equation of infinite order which is satisfied by the Lerch zeta function.

\section*{Acknowledgements}
\noindent The second author's work is supported by the Deutscher Akademischer Austausch Dienst/German Academic Exchange Service (DAAD).

\bibliography{Djida-Arran}

\bibliographystyle{unsrt}

\end{document}